\documentclass[10pt,a4paper,oneside,english]{amsart}
\usepackage[T1]{fontenc}
\usepackage[latin9]{inputenc}
\usepackage{units}
\usepackage{amsthm}
\usepackage{amssymb}

\makeatletter

\numberwithin{equation}{section}
\numberwithin{figure}{section}
\theoremstyle{plain}
\newtheorem{thm}{\protect\theoremname}[section]
  \theoremstyle{plain}
  \newtheorem{lem}[thm]{\protect\lemmaname}
  \theoremstyle{plain}
  \newtheorem{prop}[thm]{\protect\propositionname}
  \theoremstyle{remark}
  \newtheorem{rem}[thm]{\protect\remarkname}
  \theoremstyle{plain}
  \newtheorem{cor}[thm]{\protect\corollaryname}

\@ifundefined{date}{}{\date{}}

\usepackage{amsthm}

\usepackage{cite}

\providecommand{\lemmaname}{Lemma}
  \providecommand{\propositionname}{Proposition}
  \providecommand{\remarkname}{Remark}
\providecommand{\theoremname}{Theorem}

\usepackage{babel}
\providecommand{\lemmaname}{Lemma}
  \providecommand{\propositionname}{Proposition}
  \providecommand{\remarkname}{Remark}
\providecommand{\theoremname}{Theorem}

\usepackage{babel}
\providecommand{\corollaryname}{Corollary}
  \providecommand{\lemmaname}{Lemma}
  \providecommand{\propositionname}{Proposition}
  \providecommand{\remarkname}{Remark}
\providecommand{\theoremname}{Theorem}

\makeatother

\usepackage{babel}
  \providecommand{\corollaryname}{Corollary}
  \providecommand{\lemmaname}{Lemma}
  \providecommand{\propositionname}{Proposition}
  \providecommand{\remarkname}{Remark}
\providecommand{\theoremname}{Theorem}

\begin{document}

\title{Hyperplanes in the space of convergent sequences and preduals of $\ell_1$}

\author{E. Casini}

\address{Dipartimento di Scienza e Alta Tecnologia, Università dell'Insubria,
via Valleggio 11, 22100 Como, Italy }

\email{emanuele.casini@uninsunbria.it}

\author{E. Miglierina}

\address{Dipartimento di Discipline Matematiche, Finanza Matematica ed Econometria,
Università Cattolica del Sacro Cuore, Via Necchi 9, 20123 Milano,
Italy }

\email{enrico.miglierina@unicatt.it}

\author{\L. Piasecki}

\address{Institute of Mathematics, Maria Curie-Sk{\l }odowska University, pl. Marii Curie-Sk{\l }odowskiej 5 
20-031 Lublin, Poland}

\email{piasecki@hektor.umcs.lublin.pl}
\begin{abstract}
The main aim of the present paper is to investigate various structural properties
of hyperplanes of $c$, the Banach space of the convergent sequences. In particular, we give an explicit formula for the projection constants and we prove that an hyperplane of $c$ is isometric to the whole space if and only if it is $1$-complemented. Moreover, we obtain the classification
of those hyperplanes for which their duals are isometric to
$\ell_{1}$ and we give a complete description of the preduals
of $\ell_{1}$ under the assumption that the standard basis of $\ell_{1}$
is weak$^{*}$-convergent.
\end{abstract}

\subjclass[2010]{Primary: 46B45. Secondary:46B04.}

\keywords{Space of convergent sequences, Projection, $\ell_{1}$- predual,
Hyperplane. }

\maketitle

\section{Introduction}

The present paper is mainly devoted to investigate the structural
properties of the closed hyperplanes of the Banach space $c$ of the
convergent sequences. This study reveals that this class of spaces
are very interesting since it provides a complete isometric description
of the preduals of the Banach space $\ell_{1}$ when it is assumed
that the standard basis of this space is weak$^{*}$-convergent.

The starting point of our work is a result that lists some properties
about the hyperplanes of $c_{0}$. More specifically, the following
essentially known theorem summarizes some characterizations of the
$1$-complemented hyperplanes in $c_{0}$. We prefer to give a short
proof of this result for the sake of the convenience of the readers.
Indeed, some of the quoted known facts are scattered throughout the
literature and a simple remark, based on a well known property of
$\ell_{\infty}$, is easy but not immediate.
\begin{thm}
\label{thm:Hyperplanes of c_0}Let $f\in\ell_{1}$ be such that $\left\Vert f\right\Vert _{\ell_{1}}=1$.
Let us consider the hyperplane $V_{f}=\ker f\subset c_{0}$. The following
statements are equivalent 
\begin{enumerate}
\item $V_{f}$ is $1$-complemented, 
\item $V_{f}^{*}$ is isometric to $\ell_{1}$, 
\item there exists an index $j_{0}$ such that $\left|f_{j_{0}}\right|\geq\frac{1}{2}$, 
\item $V_{f}$ is isometric to $c_{0}$. 
\end{enumerate}
\end{thm}
\begin{proof} First, we recall that (1) is equivalent to (3) (see \cite{BC1974})
and (1) implies (4) since the $1$-complemented infinite dimensional subspaces of $c_{0}$
are isometric to the whole $c_{0}$ (see, e.g., \cite{LT}). Trivially,
(4) implies (2). Finally, we show that (2) implies (1). By (2) there exists
an isometry $T:V_{f}^{*}\rightarrow\ell_{1}$ hence also $T^{*}:\ell_{\infty}\rightarrow V_{f}^{**}$
is an isometry. By Proposition 5.13, p.142 in \cite{FHHM}, there
exists a norm-$1$ projection 
\[
P:\ell_{\infty}\rightarrow V_{f}^{**}.
\]

Therefore, since $V_{f}^{**}=\left[f\right]^{\bot}=\left\{ x^{**}\in\ell_{\infty}:x^{**}(f)=0\right\} $,
Corollary 2 in \cite{B1988} implies that $V_{f}$ is $1$-complemented
in $c_{0}$. \end{proof} 

One may ask whether a similar result is true when the space $c_{0}$
is replaced by $c$. Therefore, the main aim of the present paper
is to investigate the properties of hyperplanes in $c$. In particular,
we would like to determine what the implications of Theorem \ref{thm:Hyperplanes of c_0}
are preserved when we consider $c$ instead of $c_{0}$. This study
allow us to show that the behavior of hyperplanes in $c$ is much
richer than in $c_{0}$. Indeed, we will show that the counterpart
in $c$ of Theorem \ref{thm:Hyperplanes of c_0} is the following
result. \begin{thm} \label{thm:Main theorem}Let $f\in\ell_{1}$
be such that $\left\Vert f\right\Vert _{\ell_{1}}=1$ and let $W_{f}=\ker f\subset c$.
Let us consider the following properties: 
\begin{enumerate}
\item $W_{f}$ is $1$-complemented; 
\item $W_{f}$ is isometric to $c$; 
\item there exists $j_{0}\geq2$ such that $\left|f_{j_{0}}\right|\geq\frac{1}{2}$; 
\item $W_{f}^{*}$ is isometric to $\ell_{1}$; 
\item there exists $j_{0}\geq1$ such that $\left|f_{j_{0}}\right|\geq\frac{1}{2}$; 
\item $W_{f}$ is isometric to $c_{0}$; 
\item $\inf_{P}\left\Vert P\right\Vert =2$ (where $P:c\rightarrow W_{f}$
is a projection); 
\item $\left|f_{1}\right|=1,\: f_{j}=0$ for every $j\geq2$. 
\end{enumerate}
Then the following implications hold 
\[
(1)\Leftrightarrow(2)\Leftrightarrow(3)\Rightarrow(4)\Leftrightarrow(5)\Leftarrow(6)\Leftrightarrow(7)\Leftrightarrow(8).
\]

\end{thm} The previous theorem is the main result of our paper and
in order to prove it we need a number of intermediate results that
are interesting in themselves. First of all, we investigate the properties
of the projections on the hyperplanes of $c$. Indeed, in Section
\ref{sec:Projectons}, by following the approach outlined in \cite{BC1974}
for $c_{0}$, we characterize the $1$-complemented hyperplanes of
$c$ and we establish a formula to compute the projection constant
of a given hyperplane in $c$. The second step (Section \ref{sec:Isometries c and c_0})
is to study the hyperplanes of $c$ that are isometric to $c$ and
$c_{0}$ respectively. It is worth to mention that such situations
appear only when the projection constant of the hyperplane attains
respectively its minimum and maximum. Indeed we show that a hyperplane
of $c$ is isometric to $c$ itself if and only if it is $1$-complemented.
Moreover, the only hyperplane of $c$ that is an isometric copy of
$c_{0}$ has projection constant $2$ and it is the ``natural''
one, i.e. the subspace of $c$ whose elements are the vanishing sequences.
At the beginning of Section \ref{sec:Fixed-point}, Proposition \ref{thm:hyperplanes with dual equal to l_1}
characterizes the hyperplanes of $c$ whose duals are isometric copies
of $\ell_{1}$. The most interesting situations among the class of
the spaces $W_{f}$ such that its dual is $\ell_{1}$ occur when $W_{f}$
is isometric neither to $c$ nor to $c_{0}$. Therefore in these particular
cases we compute the $\sigma(\ell_{1},W_{f})$-limit of the standard
basis of $\ell_{1}$ by explicitly describing the duality between
$W_{f}$ and $\ell_{1}$ (see Theorem \ref{thm:convergence of e_n}).
This theorem allow us to obtain two interesting structural results.
First, by using a result of \cite{A1992}, we show that a $\ell_{1}$-predual
space $X$ is an isometric copy of $W_{f}$ for a suitable choice
of the functional $f\in\ell_{1}$, whenever the standard basis of
$\ell_{1}\simeq X^{*}$ is assumed to be weak$^{*}$-convergent. Second,
we characterize the hyperplanes $W_{f}$ that are $\ell_{1}$-preduals
and are isometric to a quotient of some $C(\alpha)$ where $C(\alpha)$
denotes the space of all continuous real-valued functions on the ordinals
less than or equal to $\alpha$ with the order topology.

In the sequel, whenever $X$ is a Banach space, $B_{X}$ denotes the
closed unit ball of $X$ and $\left[x\right]$ the linear span of
a vector $x\in X$. We write $X\simeq Y$ when $X$ and $Y$ are isometrically
isomorphic. We also use standard duality between $c$ and $\ell_{1}$,
that is, for $x\in c$ and $f\in \ell_{1}$: $f(x)={\displaystyle \sum_{i=0}^{\infty}}f_{i+1}x_{i}$
where $x_{0}=\lim_{i\to\infty}x_{i}$. Throughout all the paper the
hyperplane $W_{f}\subset c$ stands for the kernel of $f\in\ell_{1}$
with $\left\Vert f\right\Vert _{\ell_{1}}=1$.

\section{\label{sec:Projectons}The projections on the hyperplanes of $c$}

The aim of this section is to extend to $c$ the study of the projections
onto the hyperplanes of $c_{0}$ developed in \cite{BC1974}.

First, the following lemma establishes a formula to compute the norm
of a given projection on an hyperplane of $c$. 
\begin{lem}
\label{lem:norm projections}A projection of $c$ onto $W_{f}$ has
the form 
\[
P_{z}(x)=x-f(x)z
\]
for some $z\in f^{-1}(1)$. Moreover it holds 
\begin{equation}
\left\Vert P_{z}\right\Vert =\sup_{i\geq1}\left\{ \left|1-f_{i+1}z_{i}\right|+\left|z_{i}\right|\left(1-\left|f_{i+1}\right|\right)\right\} .\label{eq:norm of P}
\end{equation}

\end{lem}
\begin{proof} The first part is well known (see, e.g, \cite{BC1974}).
Now we will prove the formula (\ref{eq:norm of P}). We have 
\[
\left\Vert P_{z}\right\Vert =\sup_{x\in B_{c}}\sup_{i\geq1}\left|\left(P_{z}(x)\right)_{i}\right|=\sup_{x\in B_{c}}\sup_{i\geq1}\left|x_{i}-f(x)z_{i}\right|=
\]
\[
\sup_{x\in B_{c}}\sup_{i\geq1}\left|\sum_{j=0}^{+\infty}\left(\delta_{ij}-f_{j+1}z_{i}\right)x_{j}\right|.
\]
Therefore, it holds 
\begin{equation}
\left\Vert P_{z}\right\Vert \leq\sup_{i\geq1}\left(\sum_{j=0}^{+\infty}\left|\delta_{ij}-f_{j+1}z_{i}\right|\right).\label{eq:upper bound projection norm 1}
\end{equation}
Now let us consider, for every $i\geq1$, the sequences $\left\{ x^{(n,i)}\right\} _{n\geq1}\subset B_{c}$
where $x^{(n,i)}=(x_{1}^{(n,i)},x_{2}^{(n,i)},...)$ is defined by:
\[
\left\{ \begin{array}{ccc}
x_{j}^{(n,i)}={\rm sgn}\left(\delta_{ij}-f_{j+1}z_{i}\right) &  & {\rm for}\, j\leq n\\
x_{j}^{(n,i)}={\rm sgn}(-f_{1}z_{i}) &  & {\rm for}\, j>n
\end{array}\right..
\]
Then we have that, for every integers $n\geq1$, 
\[
\sup_{i\geq1}\left\Vert P_{z}\left(x^{(n,i)}\right)\right\Vert =\sup_{i\geq1}\left|\sum_{j=0}^{+\infty}\left(\delta_{ij}-f_{j+1}z_{i}\right)x_{j}^{(n,i)}\right|\geq\sup_{1\leq i\leq n}\left|\sum_{j=0}^{+\infty}\left(\delta_{ij}-f_{j+1}z_{i}\right)x_{j}^{(n,i)}\right|=
\]
\[
=\sup_{1\leq i\leq n}\left|\sum_{j=0}^{n}\left|\delta_{ij}-f_{j+1}z_{i}\right|-z_{i}{\rm sgn}(-f_{1}z_{i})\sum_{j=n+1}^{+\infty}f_{j+1}\right|.
\]
Therefore, we obtain that

\begin{equation}
\left\Vert P_{z}\right\Vert \geq\sup_{i\geq1}\left(\sum_{j=0}^{+\infty}\left|\delta_{ij}-f_{j+1}z_{i}\right|\right).\label{eq:lower bound projection norm 1}
\end{equation}
Hence, on combining (\ref{eq:upper bound projection norm 1}) and
(\ref{eq:lower bound projection norm 1}), we conclude that 
\[
\left\Vert P_{z}\right\Vert =\sup_{i\geq1}\left(\sum_{j=0}^{+\infty}\left|\delta_{ij}-f_{j+1}z_{i}\right|\right).
\]
Finally, an easy computation shows that 
\[
\left\Vert P_{z}\right\Vert =\sup_{i\geq1}\left\{ \left|1-f_{i+1}z_{i}\right|+\left|z_{i}\right|\sum_{\begin{array}{c}
j=0\\
j\neq i
\end{array}}^{+\infty}\left|f_{j+1}\right|\right\} =
\]

\[
=\sup_{i\geq1}\left\{ \left|1-f_{i+1}z_{i}\right|+\left|z_{i}\right|\left(1-\left|f_{i+1}\right|\right)\right\} .
\]

\end{proof} By means of the previous lemma, we are able to characterize
the $1$-complemented hyperplanes of $c$. 
\begin{prop}
\label{thm:1complemented in c} A norm-$1$ projection of $c$ onto
$W_{f}$ exists if and only if $\left|f_{j}\right|\geq\frac{1}{2}$
for some $j\geq2$. Moreover there exists a unique norm-$1$ projection
of $c$ onto $W_{f}$ if and only if there exists a unique index $j_{0}\geq2$
such that $\left|f_{j_{0}}\right|\geq\frac{1}{2}$.
\end{prop}
\begin{proof} By Lemma \ref{lem:norm projections} $W_{f}$ is the
rank of a norm-$1$ projection if and only if there exists $z\in c$
such that: 
\begin{eqnarray}
\left|1-f_{i+1}z_{i}\right|+\left|z_{i}\right|\left(1-\left|f_{i+1}\right|\right) & \leq & 1\quad\forall\, i\ge1\label{eq:1 Theorem norm1projection}
\end{eqnarray}

\begin{equation}
\sum_{j=0}^{+\infty}f_{j+1}z_{j}=1.\label{eq:2 Theoremnorm1projection}
\end{equation}
Inequality (\ref{eq:1 Theorem norm1projection}) implies that ${\rm sgn}(f_{i+1})={\rm sgn}(z_{i})$
for every $i\geq1$. Then (\ref{eq:1 Theorem norm1projection}) becomes
\[
1-f_{i+1}z_{i}+\left|z_{i}\right|-z_{i}f_{i+1}\leq1\quad\forall\, i\ge1
\]
and hence 
\[
\left|z_{i}\right|\left(1-2\left|f_{i+1}\right|\right)\leq0\quad\forall\, i\ge1.
\]
Therefore for every $i$ such that $\left|f_{i+1}\right|<\frac{1}{2}$
it holds $z_{i}=0.$ By equation (\ref{eq:2 Theoremnorm1projection})
we conclude that there exists at least one index $j_{0}\geq2$ such
that $\left|f_{j_{0}}\right|\geq\frac{1}{2}$.

Now let us consider an element $z^{0}\in c$ such that 
\begin{equation}
z_{j_{0}-1}^{0}=\frac{1}{f_{j_{0}}},\quad z_{j}^{0}=0\quad\forall\, j\neq j_{0}-1.\label{eq:3 Theoremnorm1projection}
\end{equation}
It is easy to see that $z^{0}$ satisfies equations (\ref{eq:1 Theorem norm1projection})
and (\ref{eq:2 Theoremnorm1projection}), $\left\Vert P_{z^{0}}\right\Vert =1$.
Finally, if there is a unique index $j_{0}$ such that $\left|f_{j_{0}}\right|\geq\frac{1}{2}$,
we remark that a unique projection $P_{z^{0}}$ exists (where $z^{0}$
is defined by (\ref{eq:3 Theoremnorm1projection})). If there are
two indexes $j_{1}$ and $j_{2}$ such that $\left|f_{j_{1}}\right|=\left|f_{j_{2}}\right|=\frac{1}{2}$
then both the projections $P_{z^{1}}$ and $P_{z^{2}}$ (where $z^{1}$
and $z^{2}$ are defined by (\ref{eq:3 Theoremnorm1projection}))
have norm $1$. \end{proof} Lemma \ref{lem:norm projections} allows
us to give an explicit formula to compute the projection constant
of the hyperplane $W_{f}$. \begin{prop} \label{prop:projection constant}Let
$f\in\ell_{1}$ be such that $\left\Vert f\right\Vert _{\ell_{1}}=1$
and $\left|f_{j}\right|<\frac{1}{2}$ for every $j\geq2$. Then 
\[
\inf_{z\in f^{-1}(1)}\left\Vert P_{z}\right\Vert =1+\left(\left|f_{1}\right|+\sum_{j=1}^{+\infty}\frac{\left|f_{j+1}\right|}{1-2\left|f_{j+1}\right|}\right)^{-1}.
\]
\end{prop}
 \begin{proof} Let us consider the quantity 
\[
\alpha_{N}=\left|f_{1}\right|+\sum_{j=1}^{N-1}\frac{\left|f_{j+1}\right|}{1-2\left|f_{j+1}\right|}+{\rm sgn}(f_{1})\sum_{j=N}^{+\infty}f_{j+1}.
\]
We first remark that there exists $N_{0}\in\mathbb{N}$ such that
for every $N\geq N_{0}$ it holds $\alpha_{N}>0$ and 
\begin{equation}
\alpha_{N}\geq\frac{\left|f_{k+1}\right|}{1-2\left|f_{k+1}\right|}\label{eq:initial prop projection constant}
\end{equation}
for every $1\leq k\leq N-1$.

Let us consider the sequence $\left\{ z^{N}\right\} _{N\geq N_{0}}\subset c$
defined by 
\[
z^{N}=\lambda_{N}\left(\underbrace{\frac{{\rm sgn}(f_{2})}{1-2\left|f_{2}\right|},...,\frac{{\rm sgn}(f_{N})}{1-2\left|f_{N}\right|}}_{N-1},{\rm sgn}(f_{1}),{\rm sgn}(f_{1}),...\right),
\]
where $\lambda_{N}$ is a positive real number such that $f(z^{N})=1$.
Therefore, it is $\lambda_{N}=\alpha_{N}^{-1}.$

Now, it holds 
\begin{equation}
\left\Vert P_{z^{N}}\right\Vert =\sup_{i\geq1}\left\{ \left|1-f_{i+1}z_{i}^{N}\right|+\left|z_{i}^{N}\right|\left(1-\left|f_{i+1}\right|\right)\right\} \leq1+\lambda_{N}\label{eq:1proposition norm projection >1}
\end{equation}
for every $N>N_{0}$. Indeed, by inequality (\ref{eq:initial prop projection constant})
we obtain that $1-\lambda_{N}\frac{\left|f_{i+1}\right|}{1-2\left|f_{i+1}\right|}\geq0$,
and hence, for $1\leq i\leq N-1$, we have 
\[
1-\lambda_{N}\frac{\left|f_{i+1}\right|}{1-2\left|f_{i+1}\right|}+\lambda_{N}\left(\frac{1-\left|f_{i+1}\right|}{1-2\left|f_{i+1}\right|}\right)=1+\lambda_{N}.
\]
Moreover, for $i\geq N$ 
\[
\left|1-\lambda_{N}f_{i+1}{\rm sgn}(f_{1})\right|+\lambda_{N}\left(1-\left|f_{i+1}\right|\right)\leq1+\lambda_{N}\left|f_{i+1}\right|+\lambda_{N}-\lambda_{N}\left|f_{i+1}\right|\leq1+\lambda_{N}.
\]
By (\ref{eq:1proposition norm projection >1}) we have that $\inf_{z\in f^{-1}(1)}\left\Vert P_{z}\right\Vert \leq1+\lambda$,
where 
\begin{equation}
\lambda=\lim_{N}\lambda_{N}=\left(\left|f_{1}\right|+\sum_{j=1}^{+\infty}\frac{\left|f_{j+1}\right|}{1-2\left|f_{j+1}\right|}\right)^{-1}.\label{eq:definition of lambda}
\end{equation}
We will finish the proof by showing that $\inf_{z\in f^{-1}(1)}\left\Vert P_{z}\right\Vert =1+\lambda$.
Let us consider two different cases.

First, let us suppose that $\left|f_{1}\right|=1$ and hence $\lambda=1$.
In this case it is well known that $\inf_{z\in f^{-1}(1)}\left\Vert P_{z}\right\Vert =2$
(see, e.g., \cite{FHHM}).

Finally, let $\left|f_{1}\right|<1$. By contradiction, let us suppose
that there exists $\hat{z}\in f^{-1}(1)$ such that 
\[
\left\Vert P_{\hat{z}}\right\Vert =\sup_{i\geq1}\left\{ \left|1-f_{i+1}\hat{z}_{i}\right|+\left|\hat{z}_{i}\right|\left(1-\left|f_{i+1}\right|\right)\right\} <1+\lambda
\]
hence 
\[
\left|1-f_{i+1}\hat{z}_{i}\right|+\left|\hat{z}_{i}\right|\left(1-\left|f_{i+1}\right|\right)<1+\lambda
\]
for every $i\geq1$ and then 
\[
1-\left|f_{i+1}\right|\left|\hat{z}_{i}\right|+\left|\hat{z}_{i}\right|-\left|\hat{z}_{i}\right|\left|f_{i+1}\right|<1+\lambda.
\]
Therefore, for every $i\geq1$, it holds that 
\begin{equation}
\left(1-2\left|f_{i+1}\right|\right)\left|\hat{z}_{i}\right|<\lambda.\label{eq:estimation norm 2}
\end{equation}
Moreover, the last relation gives that 
\begin{equation}
\left|\hat{z}_{0}\right|=\lim_{i}\left|\hat{z}_{i}\right|\leq\lim_{i}\frac{\lambda}{1-2\left|f_{i+1}\right|}=\lambda.\label{eq:estimation norm 3}
\end{equation}
Since there exists at least one index $\hat{j}\geq1$ such that $f_{\hat{j}+1}\neq0$, by using inequalities (\ref{eq:estimation norm 2}) and (\ref{eq:estimation norm 3})
and by recalling (\ref{eq:definition of lambda}), we conclude that
\[
\sum_{j=0}^{+\infty}f_{j+1}\hat{z}_{j}\leq\left|f_{1}\right|\left|\hat{z}_{0}\right|+\sum_{j=1}^{+\infty}\left|f_{j+1}\right|\left|\hat{z}_{j}\right|<\lambda\left(\left|f_{1}\right|+\sum_{j=1}^{+\infty}\frac{\left|f_{j+1}\right|}{1-2\left|f_{j+1}\right|}\right)=1.
\]
The last
inequality is a contradiction because it holds $\sum_{j=0}^{+\infty}f_{j+1}\hat{z}_{j}=1$.\end{proof}

\section{\label{sec:Isometries c and c_0}Isometries between the hyperplanes
of $c$ and the spaces $c$ and $c_{0}$}

In this section we show that the isometric structure of the hyperplanes
is completely described whenever the associated projection constant
assumes the extreme values. Indeed, we will prove that $W_{f}$ is
isometric to $c$ if and only if it is $1$-complemented, whereas
$W_{f}$ is isometric to $c_{0}$ if and only if its projection constant
is $2$. We begin to study the $1$-complemented hyperplanes of $c$.
The first step shows that a $1$-complemented hyperplane is isometric
to $c$. \begin{prop} \label{prop:isometry W c}If $W_{f}\subset c$
is $1$-complemented then $W_{f}$ is isometric to $c$.\end{prop}
\begin{proof} By Theorem \ref{thm:1complemented in c}, there exists
$j_{0}\geq2$ such that $\left|f_{j_{0}}\right|\geq\frac{1}{2}$.
Now, let us consider $T:c\rightarrow W_{f}$ defined by 
\[
T(x_{1},x_{2},...)=(x_{1},...,x_{j_{0}-2},\underbrace{\alpha}_{j_{0}-1},x_{j_{0}-1},x_{j_{0}},...)
\]
where 
\[
\alpha=-\frac{1}{f_{j_{0}}}\left(\sum_{j=0}^{j_{0}-2}f_{j+1}x_{j}+\sum_{j=j_{0}}^{+\infty}f_{j+1}x_{j-1}\right).
\]
The inverse of $T$ is $T^{-1}:W_{f}\rightarrow c$ acts on $y=(y_{1},y_{2},...)\in W_{f}$
by deleting the $\left(j_{0}-1\right)$-th component of $y$. Moreover,
if $x\in c$, then 
\[
\left|\alpha\right|\leq\frac{1}{\left|f_{j_{0}}\right|}\left(\sum_{j=0}^{j_{0}-2}\left|f_{j+1}\right|\left|x_{j}\right|+\sum_{j=j_{0}}^{+\infty}\left|f_{j+1}\right|\left|x_{j-1}\right|\right)
\]
\[
\leq\frac{1}{\left|f_{j_{0}}\right|}\left(\sum_{j=0}^{j_{0}-2}\left|f_{j+1}\right|+\sum_{j=j_{0}}^{+\infty}\left|f_{j+1}\right|\right)\left\Vert x \right\Vert=\frac{1}{\left|f_{j_{0}}\right|}\left(\sum_{\begin{array}{c}
j=0\\
j\neq j_{0}-1
\end{array}}^{+\infty}\left|f_{j+1}\right|\right)\left\Vert x \right\Vert=
\]
\[
=\frac{1}{\left|f_{j_{0}}\right|}\left(1-\left|f_{j_{0}}\right|\right)\left\Vert x \right\Vert\leq\left\Vert x \right\Vert.
\]
Therefore $T$ is an isometry between $W_{f}$ and $c$. \end{proof}
In order to prove the reverse implication we need to investigate the
family of the isometries on $c$ with $1$-codimensional range. To
this aim, we adapt to our framework some results from \cite{Gutek1991}
(Theorem 2.1 and Lemma 2.2) about the isometries on the space of continuous
functions defined on a compact set. It is worth to remark that the
mentioned results in \cite{Gutek1991} do not refer to general isometries
but they consider only shift operators. Nevertheless, by considering
the proofs of these results, it is easy to see that they hold for
general isometries with $1$-codimensional range.

As it is well known the space $c$ can be seen as the space $\mathcal{C}\left(\mathbb{N}^{*}\right)$
of continuous function on $\mathbb{N}^{*}$, where $\mathbb{N}^{*}$
denotes the Alexandroff one-point compactification of the set of positive
integers. For the sake of convenience, we denote by $0$ the unique
limit point of $\mathbb{N}^{*}$. \begin{thm} \label{thm:Gutek}(Theorem
2.1 and Lemma 2.2 in \cite{Gutek1991}). Let $T:\mathcal{C}\left(\mathbb{N}^{*}\right)\rightarrow\mathcal{C}\left(\mathbb{N}^{*}\right)$
be an isometry with $1$-codimensional range. Then there exist a closed
subset $M$ of $\mathbb{N}^{*}$, a continuous and surjective function
$\varphi:M\rightarrow\mathbb{N}^{*}$ where $\varphi^{-1}(n)$ has
at most two elements for each $n\in\mathbb{N}^{*}$and a sequence
$\left\{ \varepsilon_{n}\right\} _{n\in\mathbb{N}^{*}}$ where $\left|\varepsilon_{n}\right|=1$
such that 
\begin{equation}
\left(Tx\right)_{n}=\varepsilon_{n}x_{\varphi(n)}\quad{\rm for}\,{\rm every}\, n\in M.\label{eq:Gutek1}
\end{equation}
Moreover, only one of the two following alternatives holds: 
\begin{enumerate}
\item $M=\mathbb{N}^{*}\setminus\left\{ \bar{n}\right\} $ where $\bar{n}$
is a positive integer. In addition, in this case $\varphi$ is also
injective; 
\item $M=\mathbb{N}^{*}$and, if there exists $n'\in\mathbb{N}^{*}$ such
that the set $\varphi^{-1}(n')$ has two elements, then $\varphi^{-1}(n)$
is a singleton for every $n\in\mathbb{N}^{*}\setminus\left\{ n'\right\} $. 
\end{enumerate}
\end{thm} Now, by means of the previous theorem, we will show that
$W_{f}$ is $1$-complemented whenever it is isometric to $c$. \begin{prop} If $W_{f}\subset c$
is isometric to $c$ then $W_{f}$ is $1$-complemented.\end{prop}
\begin{proof} By recalling Proposition \ref{thm:1complemented in c},
in order to prove the theorem is sufficient to show that there exists
$j_{0}\geq2$ such that $\left|f_{j_{0}}\right|\geq\frac{1}{2}$ whenever
there exists an isometry $T:c\rightarrow W_{f}$.

First of all, by the continuity of $\varphi$ it is easy to see that
\[
\varphi(0)=0.
\]
Now, let us consider case (1) of Theorem \ref{thm:Gutek} and let us
suppose that $M=\mathbb{N}^{*}\setminus\left\{ 1\right\} $ without
loss of generality. Therefore equation (\ref{eq:Gutek1}) is true
for every $n\geq2$. Hence, all the components of $Tx$ are known
except $\left(Tx\right)_{1}=z$. Since $Tx\in W_{f}$ for every $x\in c$,
then it holds 
\begin{equation}
f_{1}\varepsilon_{0}x_{0}+f_{2}z+\sum_{j=2}^{+\infty}f_{j+1}\varepsilon_{j}x_{\varphi(j)}=0.\label{eq:Gutek2}
\end{equation}
By the injectivity of $\varphi$, we can choose $x_{N}\in c$ such
that 
\[
\varepsilon_{n}\left(x_{N}\right)_{\varphi(n)}={\rm sgn}f_{n+1}
\]
for every $2\leq n\leq N$. All the other components of $x_{N}$ are
equal to a value $x_{0}$ given by 
\[
\varepsilon_{0}x_{0}={\rm sgn}f_{1}.
\]
By (\ref{eq:Gutek2}), we have 
\[
\left|f_{1}\right|+f_{2}z_{N}+\left|f_{3}\right|+...+\left|f_{N+1}\right|+x_{0}\sum_{j=N+1}^{+\infty}f_{j+1}\varepsilon_{j}=0,
\]
where $z_{N}=\left(Tx_{N}\right)_{1}$. Since $T$ is an isometry
$\left|z_{N}\right|\leq1$ and hence, up to a subsequence, we can
suppose that $z_{N}$ converges to $\hat{z}$. Therefore, as $N\rightarrow\infty$,
it holds 
\[
\left|f_{1}\right|+f_{2}\hat{z}+\sum_{j=2}^{+\infty}\left|f_{j+1}\right|=f_{2}\hat{z}+1-\left|f_{2}\right|=0
\]
and hence $\hat{z}=-\frac{1-\left|f_{2}\right|}{f_{2}}$. Since $\left|\hat{z}\right|\leq1$,
we conclude that $\left|f_{2}\right|\geq\frac{1}{2}$.

Now, let us study the case (2) of Theorem \ref{thm:Gutek} where $M=\mathbb{N}^{*}$.
We consider three different situations. 
\begin{itemize}
\item Let $\varphi^{-1}(n)$ be a pair for an element $n\in\mathbb{N}^{*}$,
$n\neq0$. Without loss of generality, we can suppose that the map
$\varphi:\mathbb{N}^{*}\rightarrow\mathbb{N}^{*}$ is one-to-one everywhere
except at the point $1$ where it holds 
\[
\varphi^{-1}(1)=\left\{ 1,2\right\} .
\]
Since $Tx\in W_{f}$ then we have 
\begin{equation}
f_{1}\varepsilon_{0}x_{0}+f_{2}\varepsilon_{1}x_{1}+f_{3}\varepsilon_{2}x_{1}+\sum_{j=3}^{+\infty}f_{j+1}\varepsilon_{j}x_{\varphi(j)}=0.\label{eq:Gutek3}
\end{equation}
By the injectivity of $\varphi$, we can choose $x_{N}\in c$ such
that 
\[
\varepsilon_{n}\left(x_{N}\right)_{\varphi(n)}={\rm sgn}f_{n+1}
\]
for every $3\leq n\leq N$. Moreover, the component $\left(x_{N}\right)_{1}$ is chosen
to satisfies $\varepsilon_{1}\left(x_{N}\right)_{1}={\rm sgn}f_{2}$.
All the other components of $x_{N}$ are equal to the value $x_{0}$
given by 
\[
\varepsilon_{0}x_{0}={\rm sgn}f_{1}.
\]
Since $Tx_{N}\in W_{f}$, by (\ref{eq:Gutek3}), we have 
\[
\left|f_{1}\right|+\left|f_{2}\right|+f_{3}\varepsilon_{1}\varepsilon_{2}{\rm sgn}f_{2}+...+\left|f_{N+1}\right|+x_{0}\sum_{j=N+1}^{+\infty}f_{j+1}\varepsilon_{j}=0.
\]
Therefore, as $N\rightarrow\infty$, it holds 
\[
\left|f_{1}\right|+\left|f_{2}\right|+f_{3}\varepsilon_{1}\varepsilon_{2}{\rm sgn}f_{2}+\sum_{j=3}^{+\infty}\left|f_{j+1}\right|=0
\]
and hence 
\begin{equation}
1-\left|f_{3}\right|+f_{3}\varepsilon_{1}\varepsilon_{2}{\rm sgn}f_{2}=0.\label{eq:Gutek4}
\end{equation}
Similarly, by choosing $\left(x_{N}\right)_{1}$ to satisfy $\varepsilon_{2}\left(x_{N}\right)_{1}={\rm sgn}f_{3}$, we obtain an analogous equation for $f_{2}$ 
\begin{equation}
1-\left|f_{2}\right|+f_{2}\varepsilon_{1}\varepsilon_{2}{\rm sgn}f_{3}=0.\label{eq:Gutek5}
\end{equation}
Equations (\ref{eq:Gutek4}) and (\ref{eq:Gutek5}) hold true simultaneously
only if $\left|f_{2}\right|=\left|f_{3}\right|=\frac{1}{2}.$ 
\item Let $\varphi^{-1}(0)$ be a pair. Without loss of generality we can
suppose that 
\[
\varphi^{-1}(0)=\left\{ 0,1\right\} .
\]
Analysis similar to that in the previous point shows that $\left|f_{1}\right|=\left|f_{2}\right|=\frac{1}{2}.$ 
\item Finally, let $\varphi$ be injective. The same argument applied above
yields the following contradiction 
\[
\sum_{j=1}^{+\infty}\left|f_{j}\right|=0.
\]

\end{itemize}
\end{proof} Now we study the hyperplanes of $c$ that are isometric
to $c_{0}$. First, we prove that hyperplanes of $c$ with projection
constant equal to $2$ is an isometric copy of $c_{0}$. \begin{prop}
\label{prop:W isometric to c_0}If $W_{f}$ is such that $\inf_{z}\left\Vert P_{z}\right\Vert =2$
(where $P_{z}:c\rightarrow W_{f}$ is a projection) then $W_{f}$
is isometric to $c_{0}$.\end{prop} \begin{proof} Let us recall
that, by Proposition \ref{prop:projection constant}, 
\[
\inf_{z}\left\Vert P_{z}\right\Vert =1+\left(\left|f_{1}\right|+\sum_{j=1}^{+\infty}\frac{\left|f_{j+1}\right|}{1-2\left|f_{j+1}\right|}\right)^{-1}.
\]
Then, since $\inf_{z}\left\Vert P_{z}\right\Vert =2$, we have 
\[
\left|f_{1}\right|+\sum_{j=1}^{+\infty}\frac{\left|f_{j+1}\right|}{1-2\left|f_{j+1}\right|}=1.
\]
Since $\left\Vert f\right\Vert _{1}=1$, it holds 
\[
\sum_{j=1}^{+\infty}\frac{\left|f_{j+1}\right|}{1-2\left|f_{j+1}\right|}=\sum_{j=1}^{+\infty}\left|f_{j+1}\right|
\]
and hence 
\[
\sum_{j=1}^{+\infty}\left(\frac{\left|f_{j+1}\right|}{1-2\left|f_{j+1}\right|}-\left|f_{j+1}\right|\right)=\sum_{j=1}^{+\infty}\left(\frac{2\left|f_{j+1}\right|^{2}}{1-2\left|f_{j+1}\right|}\right)=0.
\]
Since $\left|f_{j+1}\right|<\frac{1}{2}$ for every $j=1,2,3,...$
(otherwise $\inf_{z}\left\Vert P_{z}\right\Vert =1$ by Proposition
\ref{thm:1complemented in c}), the last equality holds if and only
if $f_{j+1}=0$ for every $j=1,2,...$. Therefore $f=\pm(1,0,0,...)$
and hence $W_{f}\simeq c_{0}$. \end{proof} Now we prove that there
exists a unique hyperplane of $c$ isometric to $c_{0}$. This assertion
follows directly from a simple lemma that can be stated in a more
general setting. \begin{lem} Let $V$ be a subspace of $\mathcal{C}(K)$
where $K$ is a compact metric space. If $V$ is isometric to $c_{0}$
then there exists $p\in K$ such that 
\[
V\subseteq\left\{ f\in\mathcal{C}(K):f(p)=0\right\} .
\]
\end{lem} \begin{proof} Let $T:V\rightarrow c_{0}$ be an isometry.
Therefore $T^{*}:\ell_{1}\rightarrow V^{*}$ is also an isometry and
hence $T^{*}(e_{n})\in{\rm ext}B_{V^{*}}$ where $e_{n}$ is the $n$-th
element of the standard basis of $\ell_{1}$. By Lemma 6, p.441 in
\cite{DSbook}, we have 
\[
T^{*}(e_{n})=\pm x_{p_{n}}^{*}
\]
where $x_{p}^{*}$ is the evaluation functional. Since $\left\{ e_{n}\right\} $
is a weak$^{*}$-null sequence, $\left|x_{p_{n}}^{*}(f)\right|\rightarrow0$
for every $f\in V$. Now, by the compactness of $K$, there exists
an element $p\in K$ such that $f(p)=0$ for every $f\in V$. \end{proof}
If we take $\mathcal{C}(K)=\mathcal{C}(\mathbb{N}^{*})=c$ and $V=W_{f}$
in the previous lemma, then necessarily it holds 
\[
W_{f}=\left\{ x\in c:\, x_{0}=\lim_{i}x_{i}=0\right\} 
\]
since $W_{f}$ has codimension $1$.

\section{\label{sec:Fixed-point}Duality between $W_{f}$ and $\ell_{1}$
and some applications }

We start this section by characterizing the hyperplanes of $c$ such
that their duals are isometric to $\ell_{1}$. Despite its simple
proof, this result plays a central role in our approach since it allows
us conclude the proof of our main result, namely Theorem \ref{thm:Main theorem}.
\begin{prop} \label{thm:hyperplanes with dual equal to l_1}There
exists $j_{0}\in\mathbb{N}$ such that $\left|f_{j_{0}}\right|\geq\dfrac{1}{2}$
if and only if $W_{f}^{*}\simeq\ell_{1}$. \end{prop} \begin{proof}
Let us denote by $V_{f}$ the subspace of $c_{0}$ defined as $V_{f}=\ker f\subset c_{0}$.
Since 
\[
V_{f}^{*}\simeq\nicefrac{\ell_{1}}{\left[f\right]}
\]
then $W_{f}^{*}\simeq V_{f}^{*}$. By Theorem \ref{thm:Hyperplanes of c_0}
we have that $W_{f}^{*}\simeq\ell_{1}$ if and only if there exists
$j_{0}\in\mathbb{N}$ such that $\left|f_{j_{0}}\right|\geq\dfrac{1}{2}$.
\end{proof} A consequence of the main result of our paper (Theorem
\ref{thm:Main theorem}) is the classification of the hyperplanes
of $c$ such that their duals are isometric copies of $\ell_{1}.$ 
\begin{rem}
\label{Remark classification l_1 dual}The hyperplanes $W_{f}$ of
$c$ such that $W_{f}^{*}\simeq\ell_{1}$ can be divided into three
distinct classes: 
\begin{itemize}
\item $W_{f}\simeq c$ (or, equivalently, there exists $j_{0}\geq2$ such
that $\left|f_{j_{0}}\right|\geq\frac{1}{2}$); 
\item $W_{f}$ is isometric neither to $c$ nor to $c_{0}$ (or, equivalently,
$\frac{1}{2}\leq\left|f_{1}\right|<1$ and $\left|f_{j}\right|<\frac{1}{2}$
for every $j\geq2$); 
\item $W_{f}\simeq c_{0}$ (or, equivalently, $\left|f_{1}\right|=1$). 
\end{itemize}
\end{rem}
The most interesting situations occur when $W_{f}$ is isometric neither
to $c$ nor to $c_{0}$. In these cases we study the $\sigma(\ell_{1},W_{f})$-limit
of the standard basis of $\ell_{1}$. It is worth to mention that
in order to reach this aim we explicitly describe the duality between
$W_{f}$ and $\ell_{1}$. \begin{thm} \label{thm:convergence of e_n}Let
$W_{f}\subset c$ be such that $W_{f}^{*}\simeq\ell_{1}$, $\frac{1}{2}\leq\left|f_{1}\right|<1$
and $\left|f_{j}\right|<\frac{1}{2}$ for every $j\geq2$. If $\left\{ e_{n}\right\} $
is the standard basis of $\ell_{1}$, then 
\[
e_{n}\overset{\sigma(\ell_{1},W_{f})}{\longrightarrow}\hat{e},
\]
where $\hat{e}=\left(-\frac{f_{2}}{f_{1}},-\frac{f_{3}}{f_{1}},-\frac{f_{4}}{f_{1}},\dots\right).$
\end{thm} \begin{proof} Consider the map $\phi:\ell_{1}\rightarrow W_{f}^{*}$
defined by 
\[
(\phi(y))(x)=\sum_{j=1}^{+\infty}x_{j}y_{j},
\]
where $y=(y_{1},y_{2},\dots)\in\ell_{1}$ and $x=(x_{1},x_{2},\dots)\in W_{f}$.

It is easy to see that $\phi(\ell_{1})=W_{f}^{*}$, and, for any $y\in\ell_{1}$,
\begin{equation}
\left\Vert \phi(y)\right\Vert _{W_{f}^{*}}\leq\left\Vert y\right\Vert _{\ell_{1}}.\label{from below}
\end{equation}
Now, for a given $y=(y_{1},y_{2},\dots)\in\ell_{1}$, consider the
points $x^{N}$, $N=1,2,\dots$, defined as 
\[
x^{N}=\left({\rm sgn}(y_{1}),\,{\rm sgn}(y_{2}),\dots,\,{\rm sgn}(y_{N}),\, x_{0}^{N},x_{0}^{N},\dots\right),
\]
where $x_{0}^{N}$ satisfies the following equation 
\[
f_{1}x_{0}^{N}+\sum_{j=1}^{N}f_{j+1}{\rm sgn}(y_{j})+x_{0}^{N}\sum_{j=N+1}^{+\infty}f_{j+1}=0.
\]
It is clear that $x^{N}\in W_{f}$ for all $N$. Moreover, for any
$N\geq N_{0}$, where $N_{0}$ is such that $\sum_{j=N_{0}+1}^{+\infty}\left|f_{j+1}\right|<\frac{1}{2}$,
we have $x^{N}\in B_{W_{f}}$. Indeed, for every $N\geq N_{0}$, we
have 
\[
\left|x_{0}^{N}\right|=\left|\dfrac{-\sum_{j=1}^{N}f_{j+1}{\rm sgn}(y_{j})}{f_{1}+\sum_{j=N+1}^{+\infty}f_{j+1}}\right|\leq\dfrac{\sum_{j=1}^{N}\left|f_{j+1}\right|}{\left|f_{1}+\sum_{j=N+1}^{+\infty}f_{j+1}\right|}
\]
\[
=\dfrac{1-\left|f_{1}\right|-\sum_{j=N+1}^{+\infty}\left|f_{j+1}\right|}{\left|f_{1}+\sum_{j=N+1}^{+\infty}f_{j+1}\right|}\leq
\]
\[
\leq\dfrac{\frac{1}{2}-\sum_{j=N+1}^{+\infty}\left|f_{j+1}\right|}{\frac{1}{2}-\sum_{j=N+1}^{+\infty}\left|f_{j+1}\right|}=1.
\]
Now, for every $N\geq N_{0}$, we have 
\[
\left\Vert \phi(y)\right\Vert _{W_{f}^{*}}\geq\left|(\phi(y))\left(x^{N}\right)\right|=\left|\sum_{j=1}^{N}\left|y_{j}\right|+x_{0}^{N}\sum_{j=N+1}^{+\infty}y_{j}\right|.
\]
Letting $N\rightarrow+\infty$ we get 
\begin{equation}
\left\Vert \phi(y)\right\Vert _{W_{f}^{*}}\geq\left\Vert y\right\Vert _{\ell_{1}}.\label{from above}
\end{equation}
By (\ref{from below}) and (\ref{from above}) we conclude that the
map $\phi$ is an isometry.

Finally, for any $x=(x_{1},x_{2},\dots)\in W_{f}$, we obtain 
\[
\lim_{n\to\infty}(\phi(e_{n}))(x)=\lim_{n\to\infty}x_{n}=x_{0}=-\frac{1}{f_{1}}\sum_{j=1}^{\infty}f_{j+1}x_{j}=(\phi(\hat{e}))(x).
\]

\end{proof}

The last result has some interesting consequences.

First of all, by recalling Lemma 2 in \cite{A1992}, that asserts
``that the $w^{*}$-closure of the $\ell_{1}$-(standard) basis is
the only thing that is important'' to describe the structure of an
$\ell_{1}$-predual space, we obtain a complete isometric descriptions
of the preduals of $\ell_{1}$ under the additional assumption that
its standard basis is $\sigma(\ell_{1},X)$-convergent. 
\begin{cor}
Let $X$ be a Banach space such that $X^{*}=\ell_{1}$. If the standard
basis $\left\{ e_{n}\right\} $ of $\ell_{1}$ is a $\sigma(\ell_{1},X)$-convergent
sequence, then there exists $f\in\ell_{1}$ with $\left\Vert f\right\Vert _{\ell_{1}}=1$
such that $X$ is isometric to $W_{f}$.
\end{cor}
\begin{proof} Let $\hat{e}=(\hat{e}_{1},\hat{e}_{2},...)\in\ell_{1}$
be the $\sigma(\ell_{1},X)$-limit of $\left\{ e_{n}\right\} $. If
$\left|\hat{e}_{j}\right|<1$ for every $j\geq1$, then by Theorem
\ref{thm:convergence of e_n} and Lemma 2 in \cite{A1992} we have
that $X\simeq W_{f}$ where $f=(f_{1},f_{2},...)\in\ell_{1}$ is defined
by 
\[
f_{1}=\frac{1}{1+\sum_{n=1}^{\infty}|\hat{e}_{n}|},\quad f_{n}=-\frac{\hat{e}_{n-1}}{1+\sum_{n=1}^{\infty}|\hat{e_{n}}|}\:{\rm for}\,{\rm every}\, n\geq2.
\]
Now, it remains to consider the case where $\hat{e}=e_{m}$ for a
fixed $m\geq1$. Under this assumption it is easy to show that $X\simeq c$
and hence $X\simeq W_{f}$ for every $f\in\ell_{1}$ such that there
exists $j_{0}\geq2$ such that $\left|f_{j_{0}}\right|\geq\frac{1}{2}$.
\end{proof} Finally, we obtain some additional information about the isometric structure
of the hyperplanes $W_{f}$. Let us introduce the following notations:
by $\alpha$ we denote a countable ordinal and $C(\alpha)$ is the
space of all continuous real-valued functions on the ordinals less
than or equal to $\alpha$ with the order topology.
\begin{cor}
\label{cor:ordinals}There exists a countable ordinal $\alpha$ such
that $W_{f}$ is isometric to a quotient of $C(\alpha)$ if and only
if one of the following conditions holds:
\begin{enumerate}
\item there exists an index $j_{0}\geq2$ such that $\left|f_{j_{0}}\right|\geq\frac{1}{2}$;
\item $\frac{1}{2}\leq\left|f_{1}\right|\leq1$ and $\left|f_{j}\right|<\frac{1}{2}$
for all $j\geq2$ and $f=\left( f_1,f_2,...,f_n,0,0,...,0,...\right) $ for some $n\in \mathbb{N}$.
\end{enumerate}
\end{cor}
\begin{proof}
First of all, by Theorem \ref{thm:Main theorem} we have that condition
(1) is satisfied if and only if $W_{f}\simeq c$. Now, it remains to
consider the case where $\frac{1}{2}\leq\left|f_{1}\right|\leq1$
and $\left|f_{j}\right|<\frac{1}{2}$ for all $j\geq2$. Let us suppose
$f=(f_{1},f_{2},\dots,f_{n},0,0,\dots)$ for some $n\geq2$. Then
$W_{f}\subset c$ is isometric to a quotient of $C\left(\omega\cdot n\right)$.
To see it, consider the sequence of measures $\mu_{i}$ defined by
$\mu_{i}=\delta_{\omega\cdot i}$ for $i=1,2,\dots,n-1$, and, for
$i=n,n+1,n+2,\dots$, (we put $\delta_{\omega\cdot0+i}:=\delta_{i}$)
\[
\mu_{i}=-\frac{1}{f_{1}}\sum_{j=2}^{n}f_{j}\cdot\delta_{\omega\cdot(j-2)+i}+\frac{2\left|f_{1}\right|-1}{\left|f_{1}\right|\cdot i}\sum_{j=1}^{i}(-1)^{j}\cdot\delta_{\omega\cdot(n-1)+\frac{i\cdot(i-1)}{2}+j}.
\]
Now it is enough to apply Theorem \ref{thm:convergence of e_n} of
the present paper and Proposition 3 in \cite{A1992}, with the mapping
$\phi$ given by $\phi(e_{i})=\mu_{i}$. To show that condition (2)
implies that $W_{f}$ is isometric to a quotient of $C(\alpha)$ for
some $\alpha$, it is sufficient to consider Proposition 6 in \cite{A1992}
and Theorem \ref{thm:convergence of e_n} of the present paper.
\end{proof}

\end{document}